\documentclass[12pt,reqno]{amsart}
\usepackage{latexsym,amsmath}
\usepackage{amsthm}
\usepackage{amssymb}
\usepackage{amsthm}
\usepackage{graphicx}
\begin{document}
\newtheorem{theorem}{Theorem}
\newtheorem*{conjerd}{Erd\H{o}s Conjecture}
\newtheorem{fact}[theorem]{Fact}
\newtheorem{claim}{Claim}
\newtheorem{lemma}[theorem]{Lemma}
\newtheorem{cor}[theorem]{Corollary}
\newcommand\eps{\varepsilon}
\newcommand\pl{{\textrm{pl}}}
\newcommand\hV{\hat V}
\newcommand\hM{{\hat M}}
\newcommand\bV{\bar V}
\newcommand\bv{\bar v}
\newcommand\bM{{\bar M}}
\newcommand\bG{\bar G}
\newcommand\tG{\tilde G}
\newcommand\bH{\bar H}
\newcommand\Gn{\mathcal{G}_n}
\newcommand\Ge{\mathcal{G}_n(\eps)}
\newcommand\cH{\mathcal{H}}
\newcommand\cM{\mathcal{M}}
\newcommand\Cl{\textrm{Cl}}
\newcommand\Cov{\textrm{Cov}}
\newcommand{\sh}{\textbf{sh}}
\newcommand{\Sh}{\textbf{Sh}}
\title[Matchings in hypergraphs]
{On Erd\H{o}s' extremal problem\\ on matchings in hypergraphs}
\author{Tomasz \L{u}czak}
\address{Adam Mickiewicz University,
Collegium Mathematicum,
ul.~Umultowska 87,
61-614 Pozna\'n, Poland}
\email{\tt tomasz@amu.edu.pl}
\author{Katarzyna Mieczkowska}
\address{Adam Mickiewicz University,
Collegium Mathematicum,
ul.~Umultowska 87,
61-614 Pozna\'n, Poland}
\email{\tt kaska@amu.edu.pl}\thanks{The first author partially supported by 
the Foundation for Polish Science.}
\date{February 16, 2012}

\begin{abstract}
In 1965 Erd\H{o}s conjectured that the number of edges 
in $k$-uniform hypergraphs 
on $n$ vertices in which the largest matching has $s$ edges 
is maximized for hypergraphs of one of two special types. 
We settled this conjecture in the affirmative for
$k=3$ and $n$ is large enough.
\end{abstract}

\maketitle

\section{Introduction}

A {\em $k$-uniform hypergraph} or, briefly, a {\em $k$-graph} $G=(V,E)$ 
is a set of vertices $V\subseteq \mathbb{N}$ 
together with a family $E$ of 
$k$-element subsets of $V$, which are called edges.
We denote by $v(G)=|V|$ and $e(G)=|E|$ the number of vertices 
and edges of $G=(V,E)$, respectively. A family
of disjoint edges of $G$ is a {\em matching}, and by $\mu(G)$ we mean the size 
of the largest matching in $G$. In this paper we deal with the 
problem of maximizing $e(G)$ given $v(G)$ and $\mu(G)$. More formally, 
let $\cH_k(n,s)$ denote the set of all $k$-graphs $G=(V,E)$ 
such that $|V|=n$ and $\mu(G)=s$; moreover let
\begin{equation}\label{eq:deff}
\mu_k(n,s)=\max\{e(G)\,\colon\, G\in \cH_k(n,s)\},
\end{equation}
and 
\begin{equation}\label{eq:defm}
\cM_k(n,s)=\{ G\in \cH_k(n,s)\,\colon\, e(G)=\mu_k(n,s)\}\,.
\end{equation}
Let us describe two kinds of $k$-graphs from $\cH_k(n,s)$ which are
natural candidates for members of $\cM_k(n,s)$. 
By $\Cov_k(n,s)$ we denote the family of 
$k$-graphs $G_1=(V_1,E_1)$ such that $|V_1|=n$ and  
for some  subset $S\subseteq V_1$, $|S|=s$, we have
$$E_1=\{e\subseteq V_1\,\colon\, e\cap S\neq \emptyset\  \textrm{and} \ |e|=k\}\,.$$
Clearly, if $s\le n/k$, then $\Cov_k(n,s)\subseteq \cH_k(n,s)$.  
Furthermore, we define $\Cl_k(n,s)$ as 
the family of all $k$-graphs $G_2=(V_2,E_2)$ 
which  consists of a~complete subgraph on $ks+k-1$ 
and some isolated vertices, i.e. if for some subset $T\subseteq V_2$, $|T|=ks+k-1$, 
we have 
$$E_2=\{e\subseteq  T\,\colon\,   \ |e|=k\}\,.$$
Again, we have  $\Cl_k(n,s)\subseteq \cH_k(n,s)$. 
In 1965 Erd\H{o}s~\cite{E} conjectured that, indeed, the function $\mu_k(n,s)$ is 
fully determined by $k$-graphs of these two types, namely that for every 
$k$, $n$ and $s$, where $ks\le n-k+1$, the following holds 
\begin{equation}\label{eq:erdos_conj}
\mu_k(n,s)=\max\bigg\{\binom nk-\binom{n-s}k\, ,\binom {sk+k-1}k\bigg\}.
\end{equation}

Although the conjecture remains widely open a few  results have been proved 
in this direction (cf. Frankl~\cite{F}). Most of them are dealing with the case 
when $n$ is large compared to $s$, proving that
\begin{equation}\label{eq:bde1}
\cM_k(n,s)=\Cov_k(n,s)\quad \textrm{for} \quad  n\ge g(k)s,
\end{equation}
where $g(k)$ is some function of $k$. The best published bound for $g(k)$  for general $k$ 
is due to Bollob\'as, Daykin and Erd\H{o}s~\cite{BDE} who showed that~(\ref{eq:bde1}) 
holds whenever $g(k)\ge 2k^3$; recently,   Huang, Loh, and~Sudakov~\cite{HLS}
announced that~(\ref{eq:bde1}) remains true for $g(k)\ge 3 k^2$. As for the special case
of $k=3$ the current record belongs to Frankl, R\"odl and Ruci\'nski~\cite{FRR}
who verifed   (\ref{eq:bde1}) for $k=3$ and $n\ge 4s$.

The main result of this paper states that for $k=3$ and $n$ large enough 
(\ref{eq:erdos_conj}) holds for every $s$ and, moreover, the only extremal $3$-graphs 
belong to either  $\Cov_k(n,s)$ or $\Cl_k(n,s)$.

\begin{theorem}\label{thm:main}
There exists $n_0$ such that for $n\ge n_0$ large enough and each $s$, $1\le s\le (n-2)/3$, 
we have 
\begin{equation}\label{eq:erdos_conj3}
\mu_3(n,s)=\max\Big\{\binom n3-\binom{n-s}3\,, \binom {3s+2}3\Big\}\,.
\end{equation}

Furthermore, for such parameters $n$ and $s$, we have
$$\cM_3(n,s)\subseteq \Cov_3(n,s)\cup  \Cl_3(n,s)\,.$$
\end{theorem}
Let us remark that although we have made no effort to get effective bounds 
for  $n_0$, it seems to be of rather moderate order and  
it is quite conceivable that a  meticulous  analysis of cases
(possibly, with some help of computer) can give 
(\ref{eq:erdos_conj3}) for all values of $n$. Note however that the second 
part of the statement does not hold when $n=6$ and $s=1$ (or, for  
general $k$-graphs, for $n=2k$, $k\ge 3$, and $s=1$). Indeed, in this case 
$$|\cM_k(2k,1)|= 2^{\frac12\binom {2k} k},$$ 
while 
$$|\Cov_k(2k,1)|=|\Cl_k(2k,1)|=2k\,.$$

The structure of the paper goes as follows.
First  we show that if the structure of a large graph from $\cM(n,s)$ 
is `close' to a graph from  $\Cov_k(n,s)$, 
then it belongs to  $\Cov_k(n,s)$,
and the same remains true for  $\Cl_k(n,s)$. Thus, an `asymptotic version' 
of Theorem~\ref{thm:main} implies that it holds in its exact form, provided 
$n$ is large enough.  
Then,  we recall the definition and basic properties of 
the shift operation which is another important ingredient of our argument.
Finally, in the last 
part of the paper, we concentrate on the case $k=3$ and show that then 
the required asymptotic result indeed holds for shifted $3$-graphs.

\section{Stability of $\Cov$ and $\Cl$}

The aim of this section is to show that if a $k$-graph $G\in \cM_k(n,s)$ 
is, in such a way, similar  to  graphs from  $\Cov_k(n,s)$ [or $\Cl_k(n,s)$], 
then in fact it belongs to this family. In order to make it precise let us introduce
families of graphs $\Cov_k(n,s; \eps)$ and $\Cl_k(n,s;\eps)$. 
Let us recall that if $G=(V,E)$ belongs
to $\Cov_k(n,s)$, then there exists a set $S\subseteq V$, $|S|=s$, which covers all
edges of $G$. We say that $G\in \Cov_k(n,s;\eps)$ for some $\eps>0$,
if there exists a set $S\subseteq V$, $|S|=s$, which covers all but at most $\eps |E|$
edges of $G$. Moreover, we define  $\Cl_k(n,s;\eps)$ as the set of all $k$-graphs
$G$ which contain a complete subgraph on at least $(1-\eps)ks$ vertices. Then the main 
result of this section can be stated as follows.

\begin{lemma}\label{lem:types}
For every $k\ge 3$ there exist $\eps>0$ and $n_0$ such that
for every $n\ge n_0$,  $1\le s \le n/k$, and $G\in \cM_k(n,s)$ the following holds:
\begin{enumerate}
\item[(i)] if $G\in \Cov_{k}(n,s;\eps)$, then $G\in \Cov_{k}(n,s)$;
\item[(ii)]  if $G\in \Cl_{k}(n,s;\eps)$, then $G\in \Cl_{k}(n,s)$.
\end{enumerate}
\end{lemma}

Before we prove the lemma let us comment briefly on 
the formula~(\ref{eq:erdos_conj}). 
If by $s_0(n,k)$ we define the smallest $s$ for which 
$$\binom nk-\binom{n-s}k\le  \binom {ks+k-1}k\,,$$
then it is easy to see that 
$$\lim_{n\to\infty}\frac{s_0(n,k)}{n}=\alpha_k\,,$$
where $\alpha_k\in (0,1)$ is the solution of the equation
\begin{equation*}\label{eq:alpha0}
1-(1-\alpha_k)^k=k^k\alpha^k\,.
\end{equation*}
One can check that for all $k\ge 3$ we have 
\begin{equation}\label{eq:alpha}
\frac{1}{k}-\frac{1}{2k^2}< \alpha_k<\frac{1}{k}-\frac{2}{5k^2};
\end{equation}
in fact, $(1-k\alpha_k)k\to -\ln(1-e^{-1})=0.4586...$ as $k\to\infty$.

\begin{proof}[Proof of Lemma~\ref{lem:types}] In order to show (i) let us start with
the following observation.

\begin{claim}\label{cl1}
If $G\in \cM_k(n,s)$ contains a vertex $v$ which is contained in more than
$\binom {n}{k-1}-\binom {n-ks-1}{k-1}$ edges of $G=(V,E)$, then $v$ belongs to
$\binom {n-1}{k-1}$ edges of $G$. 
\end{claim}

\begin{proof} Take a vertex $v$
of large degree, and let us suppose that 
$e$ is a~$k$-subset of $V$ such that $v\in e$ and $e\notin E$. 
Then, by the definition of $\cM_k(n,s)$, the graph $G\cup e$ contains a matching $M$ 
of size $s+1$, where, clearly, $e\in M$. However, since the degree of $v$ is large, 
there exists a $(k-1)$-element subset $f\subseteq V\setminus \bigcup M$ such that
$e'=\{v\}\cup f$ is an edge of $G$. But then, $M'=M\setminus \{e\}\cup \{e'\}$ is a matching 
of size $s+1$ in $G$. This contradiction shows that each $k$-element subset of $V$
which contains $v$ is an edge of $G$.
\end{proof}
  
Now we prove (i). Let us assume that $G=(V,E)\in \cM_k(n,s)$ belongs to  $\Cov_k(n,s;\eps)$ 
and  let $S$ be the set which covers all but at most $\eps|E| $ edges of $G$.
Let $T\subseteq S$ be the set of vertices which are not  contained in 
$\binom {n-1}{k-1}$ edges of $G$ and let $t=|T|$.  We need to show that $t=0$.

Observe first that, because of (\ref{eq:alpha}), we may and shall assume that
$s\le n(1/k-2/(5k^2))$, since otherwise there exists a $k$-graph $G'\in \Cl_k(n,s)$
with more edges than $G$, contradicting the fact that $G\in \cM_k(n,s)$.
Thus, by Claim~\ref{cl1}, the number of edges in which each vertex $v\in T$ 
is contained in is at most 
$$\binom n{k-1}-\binom {n-ks-1}{k-1}\le \Big(1-\Big(\frac{2}{5k}\Big)^{k-1}\Big)\binom {n}{k-1}\,.$$
Now let $\bar G$ denote the $k$-graph obtained from $G$ by deleting all vertices 
from $S\setminus T$ and all edges intersecting with them. It is easy to see that
$\bar G\in \cM_k(n-s+t, t)$. Now, if $t\ge n/(10k^5)\ge s/(10k^4)$, 
for  any $k$-graph $\hat G\in \Cov_k(n-s+t,t)$ we have 
\begin{align*}
e(\hat G)-e(\bar G) \ge &\frac{t}k \Big(\frac{2}{5k}\Big)^{k-1} \binom {n}{k-1}-
2\eps s \binom n{k-1}\\
\ge& \Big(\Big(\frac{2}{5k}\Big)^{k-1}-20 k^5\eps \Big) \frac{t}k\binom {n}{k-1}\,.
\end{align*}
Thus,  if $\eps>0$ is small enough than $\hat G$ has more edges than $\bar G$
contradicting the fact that $\bar G\in \cM_k(n-s+t,t)$. Thus, 
 $t\le n/(10k^5)\le  (n-s+t)/k^3 $. But in such a case, Theorem~\ref{thm:main} holds 
by the result of Bollob\'as, Daykin, and Erd\H{o}s (see (\ref{eq:bde1}) above), so 
$$\bar G \in \cM_k(n-s+t)=\Cov_k(n-s+t,t)$$ 
and, since by the definition no vertex of $T$ has a full degree, $t=0$.
Consequently, $G\in Cov_k(n,s)$ and (i) follows.

Now let assume that $G=(V,E)\in \cM_k(n,s)$ belongs to $Cl_k(n,s;\eps)$.
Let $U$ be the set of vertices of the largest complete $k$-subgraph of $G$
such that $|U|\ge ks(1-\eps)$. Furthermore, 
let $M$ be a matching in $G$ of size $s$ which maximizes 
$|\bigcup M\cup  U|$, and  $M'=\{e\in M\;\colon\; e\not\subseteq U\}$. 
Then, for $n$ large enough, the following holds.

\begin{claim}\label{cl2}
\ 
\begin{enumerate}
\item[(i)] $|\bigcup M\cup U|=ks+k-1$.
\item[(ii)] $|M'|\le 2\eps ks $.
\item[(iii)] each edge of $G$ either is contained in $U$ 
or intersects an edge of~$M'$.
\end{enumerate}
\end{claim}

\begin{proof} 
Observe that at most $k-1$ vertices of $U$ can remain unsaturated by $M$, 
thus $|\bigcup M\cup U|\le ks+k-1$.
On the other hand, since $U$ induces the largest clique in $G$,
there exists a $k$-element subset $e\notin E$ such that $|e\cap U|=k-1$.
Then, since $G\in \cM_k(n,s)$, the graph $G\cup \{e\}$ contains 
a matching $M^*\cup \{e\}$ of size $s+1$. 
Thus, $M^*$ is a matching of size $s$, in which precisely 
$k-1$ vertices from $U$ are unsaturated, 
so $|\bigcup M\cup U|\ge |\bigcup M^*\cup U| \ge ks+k-1$, and (i) follows. 
To prove (ii) observe that $|M'|\leq |V(M')\setminus U|=|U\cup \bigcup M|-|U|$ and use (i),
obtaining $|M'|\le \eps ks+k-1\le 2\eps ks$ for $n$ big enough.
Finally, (iii) is a direct consequence of the choice of $M$. 
\end{proof}

Let $G'=(V,E')$ 
denote $k$-graph which consists of the clique with vertex set 
$\bigcup M\cup U$ and isolated vertices. 
Clearly, the size of the largest matching in $G'$ 
is $s$. We shall show that $G'$ has more edges than $G$ provided $|M'|>0$.
Thus, we must have $M'=\emptyset$ and the assertion follows.

In order to show that $e(G')>e(G)$ we need to introduce one more 
hypergraph. Let $H=(V\setminus U,F)$ be the hypergraph with the edge set 
$$F=\{e\cap (V\setminus U)\;\colon\; e\in E\}\,.$$
Note that $H$ is not a $k$-graph but each of its edges 
has size between $1$ and $k$. 
We call an edge $f\in F$ with $\ell$ elements {\em thick} if it
is contained in more than $3\eps k^2\binom {|U|}{k-\ell}$ edges of~$G$, 
contained entirely in $U\cup f$, and {\em thin} otherwise.
Let us make an observation somewhat analogous to Claim~\ref{cl1}. 

\begin{claim}\label{cl3}
If  an edge $f\in F$ of $\ell$-elements is thick, then 
each $k$-element subset of $U\cup f$ containing $f$ is an edge of $G$. 
\end{claim}

\begin{proof} Let us suppose that for thick $f$   
there exists an $k$-element set $e$ such that $f\subseteq e\subseteq U\cup f$ and
 $e\notin E$. 
Then, since $G\in \cM_k(n,s)$, the graph $G\cup e$ contains a matching $M''$ 
of size $s+1$, where $e\in M''$. 
Furthermore, at most $2\eps k^2s \binom{|U|}{k-l-1}\leq 3\eps k^2\binom{|U|}{k-l}$
of $(k-\ell)$-elements subsets of $U$ are covered by sets from $M''$
not contained in $U\cup f$.
Since $f$ is thick,
there exists a $(k-\ell)$-subset $h$ of $U$ which is covered
only by edges of $M''$ contained in $U$ and such that $f\cup h\in E$.
But then, one can modify $M'' \setminus  \{e\}\cup \{f\cup h\}$,
replacing edges of $M''$ which intersect $h$ by the same number of disjoint edges contained in $U$,
in such a way that the new set of edges is a matching of size $s+1$, 
contradicting the fact that $G\in \cM_k(n,s)$. 
Hence, all edges $e$ for which $f\subseteq e\subseteq U\cup f$
must already belong to $G$.
\end{proof}

Let us count edges in $|E'\setminus E|$. For every edge
$e\in M'$ consider a~vertex $v\in e\setminus U$. 
Note that $G'$ contains all $k$-element sets $e\subseteq \{v\}\cup U$,
such that $v\in e$.
Furthermore, from Claim~\ref{cl3} 
and the fact that  $U$ is the vertex set of the largest clique,
we infer that at most $3\eps k^2\binom {|U|}{k-1}$ of these sets  belong to~$G$.
Thus, 

\begin{equation}\label{eq5}
|E'\setminus E|\ge (1-3\eps k^2)|M'|\binom {|U|}{k-1}\,.
\end{equation}

Now we estimate the number of edges in $|E\setminus E'|$. 
Let us first bound the number $\gamma$ of edges $e\in E\setminus E'$ 
such that $e\cap (V\setminus U)\neq\emptyset$ is thin. Since, as we have already 
mentioned, each such edge must intersect one of $2\eps s k$ edges of $M'$,
we have 
\begin{equation}\label{eq6}
\gamma\le |M'|k \sum_{r=0}^{k-1} 3\eps k^2 \binom {n}{r}\le 3\eps |M'|k^4 \binom n{k-1}\,.
\end{equation} 
Finally, let us consider a hypergraph $H'=(W',F')$ such that $W'=(V\setminus U)\cup\bigcup M'$,
and 
$$F'=M'\cup \{e\cap W'\;\colon\; e\cap (V\setminus U)\textrm{\ is thick}\}\,.$$
It is easy to see that if the largest matching in $H'$ covers at least $k|M'|+1$ 
vertices than we can enlarge it to a matching in $G$ of size $s+1$ using 
Claims~\ref{cl2}(i) and~\ref{cl3}. 
Furthermore, if $\eps $ is small enough, 
 $|M'|\le |W'|/(2k^3)$ so one can apply the result of 
 Bollob\'as, Daykin, Erd\H{o}s~\cite{BDE} (see \ref{eq:bde1}) to infer that
\begin{equation}\label{eq7}
\begin{aligned}
|F'|&\le |M'|\bigg(\binom n{k-1}-\binom {|U|}{k-1}\bigg)+k|M'|\binom n{k-2}\\
&\le (1+3\eps k^2) |M'| \bigg(\binom n{k-1}-\binom {|U|}{k-1}\bigg)\,.
\end{aligned}
\end{equation} 
  
Thus, from (\ref{eq5}), (\ref{eq6}), and (\ref{eq7}), we get
\begin{align*}
e(G')-e(G)&\ge  (1-3\eps k^2)|M'|\binom {|U|}{k-1}- 3\eps |M'|k^4 \binom n{k-1}\\
&\quad\quad\quad\quad-(1+3\eps k^2) |M'| \bigg(\binom n{k-1}-\binom {|U|}{k-1}\bigg)\\
&\ge |M'| \bigg(2\binom {|U|}{k-1}-\binom n{k-1}-4\eps k^4 \binom n{k-1}\bigg)\,.
\end{align*} 
Due to (\ref{eq:alpha}) we may assume that $|U|/n\ge 1-1/(2k)$ 
and so 
$$\binom{|U|}{k-1}\ge 0.6\binom {n}{k-1}\,.$$
Consequently, for $|M'|>0$ we have $e(G')>e(G)$ and the assertion follows. 
\end{proof}

\section{Shifted graphs}

Let $G=(V,E)$, $V\subseteq \mathbb{N}$
be a $k$-graph. For vertices $i<j$, the graph $\sh_{ij}(G)$, 
called the $(i,j)$-shift of $G$, is obtained from $G$ 
by replacing each edge $e\in E$, such that $j\in e$,  
$i\notin e$, and $f=e-\{j\}\cup \{i\}\notin E$, by $f$.
The basic fact we shall use about $\sh_{ij}$ is that it acts nicely on families
$\cM_k(n,s)$, $\Cov_k(n,s)$ and $\Cl_k(n,s)$. 
Let us start with the
following well known result (see Frankl~\cite{F}), 
the proof of which we give here for the
completness of the argument.

\begin{lemma}\label{lem:shcm} For every $i,j$, $i<j$, 
if $G\in \cM_k(n,s)$ then $\sh_{ij}(G)\in \cM_k(n,s)$.
\end{lemma}  
 
\begin{proof}
Let us first observe that the shift operation can only decrease the size
of the largest matching. Indeed, let us assume that $M=\{e_1,\dots,e_\ell\}$
is a matching in $\sh_{ij}(G)$ but not in $G$,
and let $i\in e_1$. Then either $j\notin \bigcup_r e_r$
and so $M'=\{e_1-\{i\}\cup \{j\}, e_2,\dots,e_\ell\}$ is a matching in $G$, or 
$j\in e_2$ and then $M''=\{e_1-\{i\}\cup \{j\}, e_2-\{j\}\cup \{i\},e_3,\dots,e_\ell\}$
is a matching in $G$. Hence $\mu(\sh_{ij}(G))\le \mu(G)$ but since 
$G\in \cM_k(n,s)$ we have also $\mu(\sh_{ij}(G))=\mu(G)$.
\end{proof} 

The following simple observation will be useful in our further argument.

\begin{fact}\label{fact1}
Let $n\ge 2k-1$. If we color all $(k-1)$-element subsets of $\{1,2,\dots,n\}$
with two colors, then either we find two disjoint sets colored with different colors
or all the sets are of the same color.\hfil\qed
\end{fact}

In order to characterize the extremal graphs in $\cM_k(n,s)$
we shall use the following observation.

\begin{lemma}\label{lem:ext}
Let $n\neq 2k$, $G\in \cM_k(n,s)$, and $i<j$. 
\begin{enumerate}
\item[(i)] If $\sh_{ij}(G)\in \Cov_k(n,s)$, then  $G\in \Cov_k(n,s)$.
\item[(ii)] If $\sh_{ij}(G)\in \Cl_k(n,s)$, then  $G\in \Cl_k(n,s)$.
\end{enumerate}
\end{lemma}

\begin{proof} Let us remark first that
 if $n\le sk+k-1$,  then 
the only graph in $\cM_k(n,s)$ is the complete graph, and for $s=1$ and $n\ge 2k+1$
we have $\cM_k(n,1)=\Cov_k(n,1)$ by the extremal version of Erd\H{o}s-Ko-Rado theorem
(cf. \cite{BDE}), so we may assume that $s\ge 2$ and $n\ge 2k+1$.  
Thus, let $G=(V,E)\in \cM_k(n,s)$, $\sh_{ij}(G)\in \Cov_k(n,s)$
and let $S$ be the set which covers all edges of $\sh_{ij}(G)$. Clearly, $i\in S$.
If $j\in S$ then
$G=\sh_{ij}(G)$ so let us assume that $j\notin S$. Let us color all $(k-1)$-element subsets $f$
of $V\setminus \{i,j\}$ into two colors: red if $\{i\}\cup f\in E$ and blue if $\{j\}\cup f\in E$.
Since $S$ covers all $k$-element subsets of $V$ in $\sh_{ij}(G)$, 
each such $(k-1)$-element subsets is colored with 
exactly one color. Furthermore, if for a pair of disjoint subsets $f'$ and $f''$,
$f'$ is red and $f''$ is blue, then the edges $\{i\}\cup f'$ and $\{j\}\cup f''$ can be completed
to a matching of size $s+1$, contradicting the fact that $G\in \cM_k(n,s)$. Thus, 
by Fact~\ref{fact1}, all such sets are colored with one color and either $S$ or $S-\{i\}\cup \{j\}$
is covering all edges of $G$, i.e. $G\in \Cov_k(n,s)$. 

The proof of (ii) is very similar. We take a clique $T$ in $\sh_{ij}(G)$, $|T|=ks+k-1$, 
and observe that the only interesting case is when $i\in T$ and $j\notin T$. Then we 
color all $(k-1)$-subsets of $T$ with two colors and use Fact~\ref{fact1} 
to argue that either $T$ or $T-\{i\}\cup \{j\}$ is the clique in~$G$.
\end{proof}

Now let us define $\Sh(G)$ as a graph which is obtained from $G$ by the series of 
shifts and which is invariant under all possible shifts. Although we shall never use this fact 
it is worthy to remark that $\Sh(G)$ is uniquely determined, i.e. 
if we apply to $G$ all possible shifts 
then the resulting graph does not depend on the order the operations (see \cite{F}).
Let us state now an immediate consequence of  Lemmata~\ref{lem:shcm} and~\ref{lem:ext} 
we use directly in our proof. 

\begin{lemma}\label{lem:sh}\ 
\begin{enumerate}
\item[(i)] If $G\in \cM_k(n,s)$ then $\Sh(G)\in \cM_k(n,s)$.
\item[(ii)] If $n\neq 2k$, $G\in \cM_k(n,s)$,  
and $\Sh(G)\in \Cov_k(n,s)$, then $G\in \Cov_k(n,s)$.
\item[(iii)] If $n\neq 2k$, $G\in \cM_k(n,s)$, and   
$\Sh(G)\in \Cl_k(n,s)$, then $G\in \Cl_k(n,s)$.\qed
\end{enumerate}
\end{lemma}

\section{Proof of Theorem~\ref{thm:main}}

In this section we study the case when $k=3$. The main result 
of this part of the paper can be stated as follows.

\begin{lemma}\label{lem:main}
For every $\eps>0$ there exists $n_0$ such that for every $n\ge n_0$,  $1\le s\le n/3$, 
and $G\in \cM_3(n,s)$ we have 
$$\Sh(G)\in \Cov_3(n,s;\eps)\cup \Cl_3(n,s;\eps)\,.$$
\end{lemma}

We shall show Lemma~\ref{lem:main} by a detailed analysis of the structure 
of $\Sh(G)$ but before we do it let us observe that it implies Theorem~\ref{thm:main}.

\begin{proof}[Proof of Theorem~\ref{thm:main}]
Let $G\in \cM_3(n,s)$. Then, by Lemma~\ref{lem:sh}(i), $\Sh(G)\in \cM_3(n,s)$. 
Thus, from Lemmata~\ref{lem:types} and \ref{lem:main}, for $n$ large enough we get
$$\Sh(G)\in  \Cov_3(n,s)\cup \Cl_3(n,s)\,,$$
and so, by Lemma~\ref{lem:sh}(ii),(iii)
$$\hfill G\in \Cov_3(n,s)\cup \Cl_3(n,s)\,.\hfill\qed$$
\renewcommand{\qed}{} 
\end{proof}

Let us remark that in order to show Theorem~\ref{thm:main} it is enough to 
show Lemma~\ref{lem:main} for some given $\eps>0$. 

\begin{proof}[Proof of Lemma~\ref{lem:main}]
Let $\eps >0$ and $G\in \cM_3(n,s)$. By Lemma~\ref{lem:sh}(i), $\Sh(G)\in \cM_3(n,s)$. 
To simplify the notation, by writing $(i,j,k)$ 
we always mean that an edge $\{i,j,k\}$ is such that $i<j<k$.
Let $M=\{(i_l,j_l,k_l):l=1,\ldots,s\}$ be the largest matching in $\Sh(G)$,
and let partition its vertex set into three parts  $V(M)=I\cup J\cup K$ such that 
for every edge $(i,j,k)\in M$ we have $i\in I$, $j\in J$, and $k\in K$. 
Moreover, let vertices of $K$ be labeled in such a way that $k_l<k_m$ for every $l<m$, 
and denote $L=\{i_l, j_l,k_l\in V(M): l\leq (1-\eps)s\}$.
We shall show that for $n$ large enough either $I$ covers all but at most $\epsilon|E|$ edges of $\Sh(G)$
or $\{e\in \Sh(G):e\subset L\}$ is a clique.

In order to study the structure of $\Sh(G)$ we  introduce an auxiliary hypergraph $H$.
Denote by $V'$ the set of vertices which are not saturated by $M$.
Obviously, none of the edges of $\Sh(G)$ is contained in $V'$.
We use $\deg_{V'}(v)$ to denote the number of pairs $u,w\in V'$ 
such that $\{v,u,w\}$ is an edge in $\Sh(G)$. 
Similarly, a number of vertices $w\in V'$ such that $\{v,u,w\}\in \Sh(G)$ is denoted $\deg_{V'}(v,u)$.
Finally, we use $e(v)$ to denote an unique edge of $M$ containing vertex~$v$.
Let  $H=(W,F)$ be a hypergraph with vertices $W=V(M)$ and the edge set $F=M \cup F_1\cup F_2\cup F_3$, where
\begin{eqnarray*}
F_1&=&\{v\in W: \deg_{V'}(v)\geq 20 n\},\\
F_2&=&\{\{v,w\}\in W^{(2)}: e(v)\neq e(w) \text{ and } \deg_{V'}(v,w)\geq 20\},\\
F_3&=&\{\{v,w,u\}\in W^{(3)}: e(v), e(w) \text{ and } e(u) \text{ are pairwise different}\}.
\end{eqnarray*}
Note that since $\Sh(G)$ is shifted, hypergraphs $F_1, F_2, F_3$ are shifted as well.
We shall call an edge $e$ of $\Sh(G)$ \emph{traceable} if $e\cap V(M)\in F$, and 
\emph{untraceable} otherwise.  Observe also that the number of untraceable edges
of $\Sh(G)$ is bounded from above by $31n^2$, so we can afford to ignore them.

We call a triple $T$ of edges from $M$ \emph{bad},
if in $\bigcup T$ there are three disjoint edges of $H$ whose union 
intersects $I$ on at most 2 vertices, and \emph{good} otherwise.
We show first that there are only few {bad} triples in~$M$. 

\begin{claim}\label{cl:bad}
No three disjoint triples are bad. 

Consequently, 
there exist at most six edges in the matching $M$ so that 
each bad triple contains one of these edges.
\end{claim}

\begin{proof}
Let us suppose that there exist nine disjoint edges 
$\{(i_l,j_l,k_l):l=1,\ldots,9\}\subset M$ 
such that in $\{i_l,j_l,k_l:l=1,\ldots,9\}$ one can find 
a set of nine disjoint edges $H'\subset H$,
which do not cover vertices $i_3$, $i_6$ and $i_9$.
One can easily see that for any ordering of the sets $\{j_3, j_6, j_9\}$ 
and $\{k_3, k_6, k_9\}$ there exists a permutation $\sigma(3)$, $\sigma(6)$, $\sigma(9)$
such that $j_{\sigma(9)}> j_{\sigma(6)}$ and $k_{\sigma(9)}> k_{\sigma(3)}$; 
to simplify the notation let us assume that $j_9> j_6>i_6$ and $k_9>k_3>i_3$.
Replace in $H'$  an edge $e$ which contains  $j_9$ by $e'=e\setminus \{j_9\}\cup \{i_6\}$
and the edge $f$ containing $k_9$ by $f'=e\setminus \{k_9\}\cup \{i_3\}$; note that both
$e'$ and $f'$ belong to $H$ since $H=\Sh(H)$. 
Thus, we obtain the family of 
nine disjoint edges of $H''\subseteq H$, all of which are contained in 
eight edges of $M$. Furthermore, since edges from $F_1\cup F_2$ 
have large degrees,  all edges from $H''$ which belong to $F_1\cup F_2$ 
can be simultaneously extended to disjoint edges of $\Sh(G)$ by 
adding to them vertices from $V\setminus \bigcup M$. But this would lead to 
a matching  $M'$ of size $s+1$ in $\Sh(G)$ 
contradicting  the assumption $\Sh(G)\in \cM_3(n,s)$.
\end{proof}

Now we study properties of {good} triples.
We start with the following simple observation.

\begin{claim}\label{cl12}
Let $T$ be a good triple. 
\begin{itemize}
\item[(i)] $\left( F_1\cap \bigcup T\right)\subset I$.
\item[(ii)] For any two edges of $T$ there are 
at most $5$ edges in $F_2$ contained in their vertex set.

Moreover, the only possible configuration with exactly $5$ edges
from $F_2$  
is  when  all these edges intersect $I$ (see Fig.~1).
\end{itemize}
\begin{center}
\includegraphics[scale=0.1]{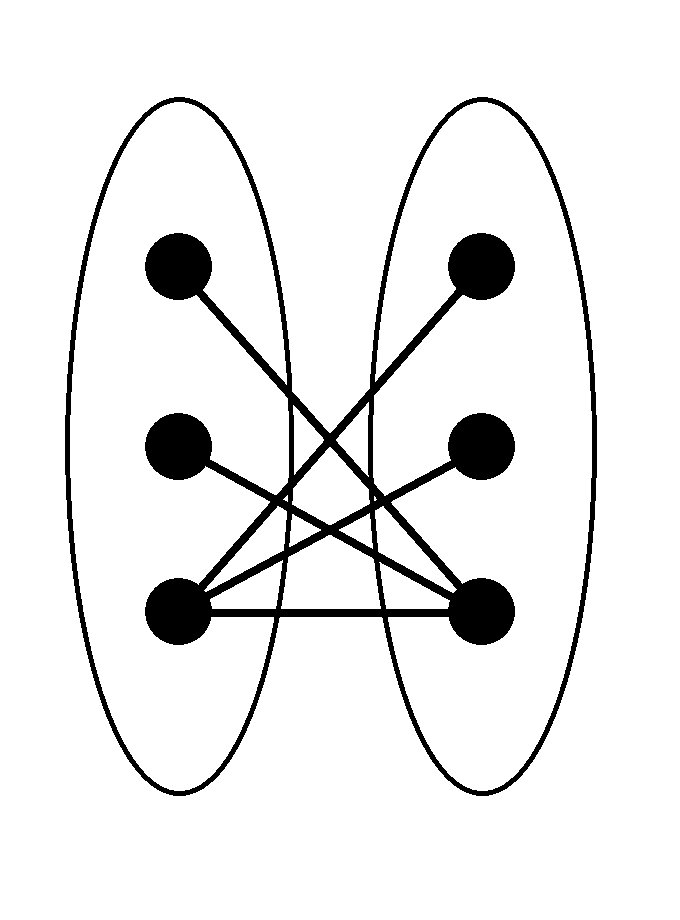} \\ Fig. 1. \\ 
\end{center}
\end{claim}
\begin{proof}
Let $T=\{(i_1,j_1,k_1),(i_2,j_2,k_2),(i_3,j_3,k_3)\}$ 
be a {good} triple.

(i) Let $j_1<j_2<j_3$ and assume 
that $\left( F_1\cap \bigcup T\right)\not\subset I$.
Then, since hypergraph $F_1$ is shifted, $\{j_1\}\in F_1$ and
$T$ is a~{bad} triple because of the edges $\{j_1\}$, $(i_2,j_2,k_2)$, $(i_3,j_3,k_3)$,
a~contradiction.

(ii) Let assume by contradiction that 
$6$ edges from $F_2$ are contained in $\{i_1,j_1,k_1,i_2,j_2,k_2\}$.
Then $\{j_1,j_2\}\in F_2$ and at least one of the edges 
$\{i_1,k_2\}, \{i_2,k_1\}$ is in $F_2$. 
Let us assume that $\{i_1,k_2\}\in F_2$. 
Then, $T$ is {bad} because of the edges $\{j_1,j_2\},\{i_1,k_2\},(i_3,j_3,k_3)$.
\end{proof}

For a triple $T\in M^{(3)}$ and for $i=1,2,3$, let $f_i(T)$ 
be the number of edges of $F_i$ contained in $\bigcup T$. 
Clearly, $f_1(T)\leq 9$, $f_2(T)\leq 27$ and $f_3(T)\leq 27$ for any triple $T$.
However, if $T$ is {good}, then, by Claim~\ref{cl12}, 
we immediately infer that $f_1(T)\leq 3$ and $f_2(T)\leq 15$.
Our next result shows how to bound $f_1(T)$ and $f_2(T)$ more precisely
for {good} triples for which  $f_3(T)$ is large.

\begin{claim}\label{cl33}
Let $T$ be  a {good} triple.
\begin{itemize}
\item[(i)] If $f_3(T)\geq 24$, then $f_1(T)=f_2(T)=0$.
\item[(ii)] If $f_3(T)= 20$, then $f_1(T)\leq 1$ and $f_2(T)\leq 12$.
\item[(iii)] If $f_3(T)\leq 19$, then $f_1(T)\leq 3$ and $f_2(T)\leq 15$. 

Moreover,
the only triples for which  $f_3(T)=19$,  $f_2(T)=15$, and $f_1(T)=3$, are those in which 
each edge of $H$ contained in $\bigcup T$ intersects $I$.
\item[(iv)] If $f_3(T)=21$, then $f_1(T)\leq 1$ and $f_2(T)\leq 10$.
\item[(v)] If $22\leq f_3(T)\leq23$, then $f_1(T)=0$ and $f_2(T)\leq 7$.
\end{itemize}
\end{claim}

\begin{proof}
Let $T=\{(i_1,j_1,k_1),(i_2,j_2,k_2),(i_3,j_3,k_3)\}$ 
be a {good} triple.

(i) Observe that since $f_3(T)\geq 24$, 
one of the following pairs of edges must be in~$H$.
\begin{center}
\includegraphics[scale=0.1]{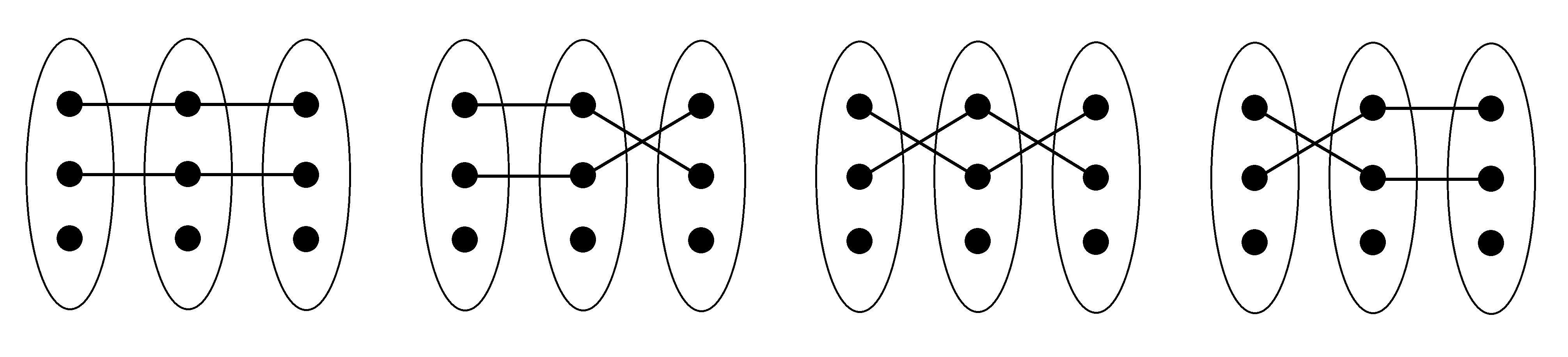}\\Fig. 2.\\ 
\end{center}
Let $e,f\in F_3$ be disjoint edges such that $e,f\subset\{j_1,j_2,j_3,k_1,k_2,k_3\}$,
and let us assume that $i_1<i_2<i_3$. 
If $f_1(T)\neq 0$, then $i_1\in F_1$ and so $T$ is {bad} because of $\{i_1, e,f\}$.
Similarly, if $f_2(T)\neq 0$, then $\{i_1,i_2\}\in F_2$ 
and again $T$ is {bad}, while we assumed that $T$ is {good}.

(ii) Observe that if  $\{j_1,j_2,j_3\}\notin F_3$, 
then every edge contained in $\{j_1,j_2,j_3,k_1,k_2,k_3\}$
is not in $F_3$. Since there are $8$ such edges,
we have $f_3(T)\leq 19$. 
Thus, if $f_3(T)\geq 20$, then $\{j_1,j_2,j_3\}\in F_3$, 
and because $T$ is {good}, it is easy to see that $f_1(T)\leq 1$. 
Now assume by contradiction that $f_2(T)\geq 13$. 
Then, there are two edges in $T$, let say $(i_1,j_1,k_1), (i_2,j_2,k_2)$, 
such that at least five edges of $F_2$ are contained in theirs set of vertices.
By Claim~\ref{cl12}, $\{i_1,k_2\}, \{k_1,i_2\}\in F_2$
and thus, $T$ is {bad} because of the edges $\{i_1,k_2\}, \{k_1,i_2\}, \{j_1,j_2,j_3\}$.

(iii) It is a direct consequence of Claim~\ref{cl12} and the fact that  $\{j_1,j_2,j_3\}\notin F_3$,
since then $f_2(T)\le 12$ (see (ii) above).

(iv) Since $f_3(T)=21$, we know that $\{j_1,j_2,j_3\}\in F_3$ and $\{k_1,k_2,k_3\}\notin F_3$.
Therefore, at least one pair of edges from Fig. 3. and Fig. 4. is in~$F_3$.
\begin{center}
\includegraphics[scale=0.1]{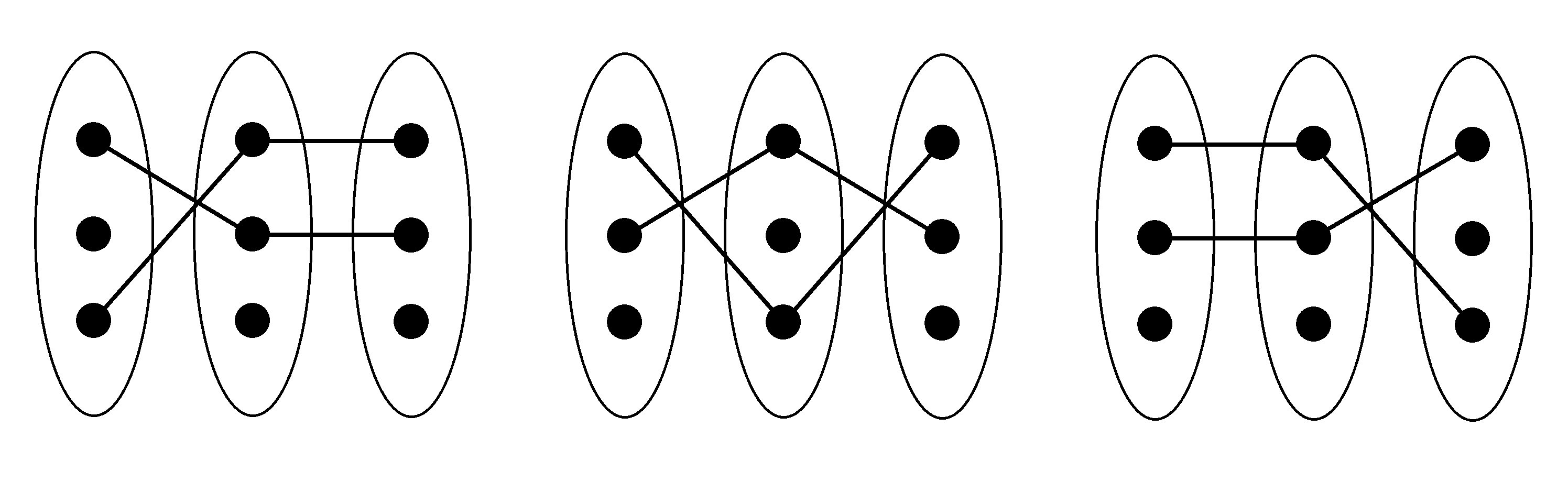}\\Fig. 3.\\ 
\end{center}
\begin{center}
\includegraphics[scale=0.1]{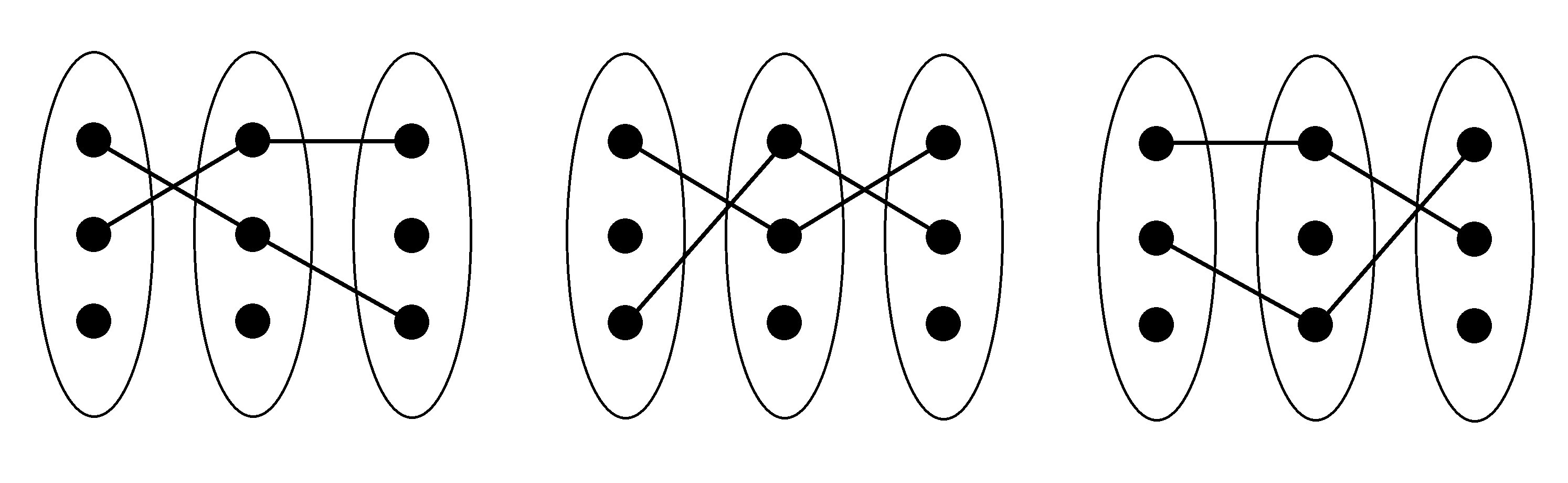}\\Fig. 4.\\ 
\end{center}
Since such a pair saturates only one vertex from $I$, we have $f_1(T)\le 1$.
To estimate $f_2(T)$ let assume that $j_2$ is not saturated by this pair of edges.
Then, $\{i_1,j_2\}, \{j_2,i_3\}\notin F_2$, because $T$ is {good}.
Consequently, $\{i_1,k_2\}, \{j_1,j_2\}$, $\{j_2,j_3\}, \{k_2,i_3\}$
are also not in $F_3$ and thus, at most six edges of $F_2$ 
are contained in $\{i_1,j_1,k_1,i_2,j_2,k_2\}$ or in $\{i_2,j_2,k_2,i_3,j_3,k_3\}$.
Now, since $\{j_1,j_2,j_3\}\in F_3$, using the same argument as in (ii),  
we conclude that at most four edges of $F_2$ are contained in $\{i_1,j_1,k_1,i_3,j_3,k_3\}$.
Hence, $f_2(T)\leq 10$.

(v) From (i) we know that if in $T$ we can find
one of the pairs of edges marked on Fig. 2, then $f_1(T)=f_2(T)=0$. 
Thus, let  assume that for each of these pairs at least one edge is not in $F_3$
and $22\leq f_3(T)\leq 23$. Hence $\{j_1,j_2,j_3\}\in F_3$ 
and $\{k_1,k_2,k_3\}\notin F_3$.
Now consider $\{j_1,j_2,k_3\}$, $\{j_1,k_2,j_3\}$, 
$\{k_1,j_2,j_3\}$. It is easy to check that if at most one of them 
is in $F_3$, then $f_3(T)\le 21$. Thus, we split our further argument into
two cases.

\medskip

{\it Case 1.  All three edges $\{j_1,j_2,k_3\}$, $\{j_1,k_2,j_3\}$, 
$\{k_1,j_2,j_3\}$ are in $F_3$.} 

\smallskip

Then, $\{j_1,k_2,k_3\}$, 
$\{k_1,j_2,k_3\}$, $\{k_1,k_2,j_3\}$, $\{k_1,k_2,k_3\}\notin F_3$.
Therefore, as $f_3(T)\ge 22$, at least 
two pairs of edges shown on Fig. 3. are in~$F_3$.
Let say these are $\{i_1,k_2,k_3\}, \{k_1, j_2,j_3\}$ 
and $\{k_1,k_2,i_3\}, \{j_1,j_2,k_3\}$. 
Since $T$ is good, edges 
$\{j_1,i_2\}$, $\{j_1,i_3\}$, $\{i_2,j_3\}$, $\{i_1,j_3\}$ are not in $F_2$,
and because $F_2$ is shifted, the edges of $F_2$ 
contained in $\bigcup T$ are contained in the set
$\{\{i_1,i_2\}$, $\{i_1,j_2\}$, $\{i_1,k_2\}$, $\{i_2,i_3\}$, $\{j_2,i_3\}$, $\{k_2,i_3\}$, $\{i_1,i_3\}\}$.
Hence, $f_2(T)\leq 7$.
It is also easy to observe that in that case $f_1(T)=0$.

\medskip

{\it Case 2. Exactly two of the edges $\{j_1,j_2,k_3\}$, 
$\{j_1,k_2,j_3\}$, $\{k_1,j_2,j_3\}$ are in $F_3$.} 

\smallskip

Without loss of generality let $\{j_1,j_2,k_3\}$, $\{j_1,k_2,j_3\}\in F_3$.
Then, $\{k_1,j_2,j_3\}$, $\{k_1,j_2,k_3\}$, $\{k_1,k_2,j_3\}$, $\{k_1,k_2,k_3\}\notin F_3$.
Therefore, if $f_3(T)=23$, then all other edges are in $F_3$, 
and so two pairs of edges shown on Fig. 3. are in $F_3$. 
Thus, as we have shown in the proof of Case 1, $f_2(T)\leq 7$.
Let now consider the case when $f_3(T)=22$.
If both pairs of edges $\{j_1,k_2,j_3\}, \{k_1,i_2,k_3\}$ and 
$\{k_1,k_2,i_3\}$, $\{j_1,j_2,k_3\}$ are in $F_3$, then again $f_2(T)\leq 7$.
Let now assume that only one of these pairs is in $F_3$,
let say $\{j_1,k_2,j_3\}, \{k_1,i_2,k_3\} \in F_3$. 
Then also a pair $\{j_1,k_2,k_3\}, \{k_1,j_2,i_3\}$ is in $F_3$.
Thus, $\{i_1,j_2\}$, $\{j_2,i_3\}$, $\{i_1,j_3\}$, $\{i_2,j_3\}\notin F_2$,
and therefore, $f_2(T)\leq 7$.
In that case we also have $f_1(T)=0$.
\end{proof}

Now we  bound the number of edges in $\Sh(G)$. First of all let us 
remove from $M$ six edges so that in the remaining matching $\bar M$ 
we have only good triples (see Claim~\ref{cl:bad}). In this way we omit at most $9n^2$ 
edges of $\Sh(G)$. Let us recall also that the number of untraceable 
edges of $\Sh(G)$ is at most $31 n^2$. 
Finally, observe that for each edge $f\in F_i$ there are at most
$\binom{n-3s}{3-i}$ edges $e\in \Sh(G)$ such that $e\cap V(M)=f$.
Thus,  the number of edges 
in $\Sh(G)$ is bounded from above by
$$e(\Sh(G))\leq   |F_1|\binom{n-3s}{2}+|F_2|(n-3s)+|F_3|+ 40 n^2.$$
To bound $|F_i|$, let us sum $f_i(T)$ over all $T\in {\bar M}^{(3)}$.
Observe  that in such a~sum each edge from $F_i$ is cour nted 
exactly $\binom{s-i}{3-i}$ times.
Thus,
$$e(\Sh(G)) \leq \sum_{T\in {\bar M}^{(3)}}\left({f_1(T)\frac{(n-3s)^2}{s^2}} 
+ {f_2(T)\frac{n-3s}{s}} + f_3(T) \right) + 41 n^2.$$
Now we divide {good} triples into $27$ groups,
depending on $f_3(T)$. 
If  $T_i=\{T\in {\bar M}^{(3)}: f_3(T)=i\}$ 
for $i=1,\ldots,27$, then
$$ e(\Sh(G)) \leq \sum_{i=1}^{27}{\sum_{T\in T_i}{\left({f_1(T)\frac{(n-3s)^2}{s^2}} 
+ {f_2(T)\frac{n-3s}{s}} + f_3(T) \right) }} + 41 n^2.$$
Let now denote $x_1=\sum_{i=1}^{19}{|T_i|}$, $x_2=|T_{20}|$, 
$x_3=|T_{21}|$, $x_4=|T_{22}|+|T_{23}|$, $x_5=\sum_{i=24}^{27}{|T_i|}$.
By Claim~\ref{cl33}, we get the following bound.
\begin{eqnarray*}
e(\Sh(G)) &\leq & (3 x_1 + x_2 + x_3) \frac{(n-3s)^2}{s^2} \\
&+& (15 x_1 + 12 x_2 + 10 x_3 + 7 x_4) \frac{n-3s}{s} \\
&+& (19 x_1 + 20 x_2 + 21 x_3 + 23 x_4 + 27 x_5) + 41 n^2.
\end{eqnarray*}
Now it is sufficient to maximize the above function
under the conditions $\sum_{i=1}^5{x_i}\leq \binom {s-6}3$ and 
$x_i\geq 0$ for every $i=1,\ldots,5$.
Then, we are to maximize a function 
$$f_{s,n}(x_1,x_2,x_3,x_4,x_5)=\sum_{i=1}^5\alpha_i(s,n) x_i,$$
where 
\begin{align*}
\alpha_1(s,n)&= 3(n-3s)^2/s^2+15(n-3s)/s+19\\
\alpha_2(s,n)&= (n-3s)^2/s^2+12(n-3s)/s+20\\
\alpha_3(s,n)&= (n-3s)^2/s^2+10(n-3s)/s+21\\
\alpha_4(s,n)&= 7(n-3s)/s+23\\
\alpha_5(s,n)&= 27\,,
\end{align*}
over domain $\sum_{i=1}^5x_i\le \binom{s-6}3$, $x_i\ge 0$ for $i=1,2,\dots,5$.
This is a linear function of $x_i$'s, so in order to maximize it it 
is enough to check which of the coefficients $\alpha_i(s,n)$ is the largest 
one and set the variable $x_i$ which corresponds to this coefficient 
to be maximum, while the rest of the variables should be equal to zero.

It is easy to check that if $s=an$ and $a<a_0$, where $a_0=(\sqrt{321}-3)/52$, 
then  $\alpha_1(s,n)$ dominates, and so for $s=an$, $a<a_0$,
we have 
$$ e(\Sh(G))\leq \frac{1}{6} (3 s - 3 s^2 + s^3)+O(n^2),$$
what nicely matches the lower bound for $e(\Sh(G))$ 
given by
$$e(\Sh(G))\ge \binom n3 - \binom {n-s}3=\frac{1}{6} (3 s - 3 s^2 + s^3)+O(n^2)\,.$$
Furthermore, in order to achieve this bound for all but $O(n^2)$ 
triples $T$ we must have 
$f_3(T)=19$, $f_2(T)=15$, $f_1(T)=3$, which is possible only if all edges of such 
triple intersect $I$ (see Claim~\ref{cl33}(iii)). Consequently, for this range 
of $s$, in $\Sh(G)$ there is a subset $I$, $|I|=s$, which covers all but at most 
$O(n^2)$ edges of $\Sh(G)$.

For $a>a_0$ the dominating coefficient is $\alpha_5(s,n)=27$, which gives 
$$e(\Sh(G))\le \frac{9}{2} s^3+O(n^2)\,,$$
matched by the lower bound
$$e(\Sh(G))\ge \binom {3s+2}{3}=\frac{9}{2} s^3+O(n^2)\,.$$
Again to achieve this bound for all but $O(n^2)$ triples we must have $f_3(T)=27$.
But then the largest independent set contained in $\bigcup M$ has at most 
$O(n^{2/3})$ vertices and so, because of shifting, $\bigcup M$ contains a clique 
of size at least $s-O(n^{2/3})=s-O(s^{2/3})$.

In order to complete the proof we need to consider the remaining case when 
$s=(a_0+o(1)) n$. Since $\alpha_1(a_0n,n)=\alpha_5(a_0n,n)>\alpha_i(a_0n,n)$ for $i=2,3,4$,
we infer that in $\Sh(G)$ all  triples, except of at most $O(n^2)$, must 
be of one of two types: either for such a triple $T$ we have $f_3(T)=27$, $f_2(T)=f_1(T)=0$, 
or $f_3(T)=19$, $f_2(T)=15$, $f_3(T)=3$ and all edges of $H$ contained in $\bigcup T$ intersect $I$. 
It is easy to see that  it is possible only when one of these two types of triples dominate. 
Indeed, let $M'\subseteq M$ denote the set of edges of$M$ which contain a singleton edge
form $F_1$. Since the number of triples which are  contained neither in $M'$, 
nor in $M\setminus M'$ is $O(s^2)$, so $\min\{|M'|,|M\setminus M'|\}$ is bounded and,
consequently, all but $O(s^2)=O(n^2)$ triples must be of one of our 
two types.
Hence
\begin{align*}
\Sh(G)&\in \Cov_3(n,s;O(n^{-1}))\cup \Cl_3(n,s;O(n^{-1/3}))\\
&\subseteq \Cov_3(n,s;\eps)\cup \Cl_3(n,s;\eps)\,,
\end{align*}
and the assertion follows.
\end{proof}

\bibliographystyle{plain}


\end{document}